\documentclass[11pt]{amsart}
\usepackage{amssymb,amsmath,amsthm,anysize}
\usepackage{longtable}
\usepackage{booktabs}
\usepackage{multirow,bigstrut,bigdelim,color}

\usepackage{multicol}

\usepackage{lscape}
\usepackage{enumerate}

 \usepackage{tikz}
\usetikzlibrary{graphs}
\usetikzlibrary{graphs.standard}

\usepackage{pgf,tikz}
\usepackage{mathrsfs}
\usetikzlibrary{arrows}
\usepackage{wrapfig}
 \usepackage{tikz}
\usetikzlibrary{graphs}
\usetikzlibrary{graphs.standard}
\usetikzlibrary{calc,intersections,through,backgrounds}
\usepackage{tkz-euclide}

\usepackage{graphicx}
\usepackage{amscd}
\usepackage{url}
\usepackage{eurosym}

\usepackage{verbatim}

\definecolor{ududff}{rgb}{0.30196078431372547,0.30196078431372547,1}

\newtheorem{theorem}{Theorem}[section]
\newtheorem{cor}[theorem]{Corollary}
\newtheorem{definition}[theorem]{Definition}

\newtheorem{lemma}[theorem]{Lemma}
\newtheorem{result}[theorem]{Result}
\newtheorem{remark}[theorem]{Remark}

\newcommand{\wt}{\mathrm{wt}}

\newcommand{\cE}{\mathcal{E}}

\newcommand{\bF}{\mathbb{F}}

\newcommand{\AG}{\mathrm{AG}}

\newcommand{\PG}{\mathrm{PG}}

\newcommand{\F}{\mathbb{F}}

\title{Locally repairable codes with high availability based on generalised quadrangles}

\author{Michel Lavrauw \and Geertrui Van de Voorde}\thanks{This research was made possible by an Erskine Fellowship from the University of Canterbury, Christchurch, New Zealand.}\thanks{The first author acknowledges the support of {\em The Scientific and Technological Research Council of Turkey}, T\"UB\.{I}TAK (project no. 118F159).}

\begin{document}

\maketitle

{\bf {\it Abstract} Locally Repairable Codes (LRC's) based on generalised quadrangles were introduced by Pamies-Juarez, Hollmann and Oggier in  \cite{PaHoOg2013}, and bounds on the repairability and availability were derived. In this paper, we determine the values of the repairability and availability of such LRC's for a large portion of the currently known generalised quadrangles. In order to do so, we determine the minimum weight of the codes of translation generalised quadrangles and characterise the codewords of minimum weight.
}

\section{Introduction}

Locally recoverable/repairable codes (LRC's) have been designed to provide reliability with small repair traffic
in systems like cloud storage and distributed computing where high data availability is necessary \cite{lrc1,lrc2}. The dual codes of {\em partial geometries} were used to construct LRC's with high availability in \cite{PaHoOg2013} where bounds on the repairability and the availability for these codes were derived. In the same paper (see \cite[Theorem 1 and Remark 1]{PaHoOg2013}) it was shown that the maximum {\em rate} of a balanced locally repairable binary code from a partial geometry is achieved when the partial geometry is a {\em generalised quadrangle}.  
In this paper, we determine the exact values of the {\em repair degree} and the {\em repair availability} of any LRC constructed from a
classical or translation generalised quadrangle or a generalised quadrangle $T_2^*(O)$, $O$ a hyperoval (see Theorem \ref{main2}). 
In order to do so, we will study the minimum weight of the linear codes generated by the incidence matrix of generalised quadrangles, where we focus on the case of {\em translation generalised quadrangles} (see Theorem \ref{main}).

The study of minimum weight codewords in linear codes is a classical problem in coding theory. Within this topic, the codes generated by the incidence matrix of points and blocks of certain incidence structures are of special interest. The codes generated by the incidence matrix of points and subspaces of affine and projective spaces are well understood since the 1970's. 
For partial geometries, the {\em dual} of these codes have attracted attention, not only as LRC's, but also because of their properties when seen as a Low Density Parity Check code (LDPC code). The minimum weight has been determined in several cases \cite{johnsonweller,kim,pepe}. However, when the partial geometry is not fully embedded in an affine or projective space, far less is known. In this paper, we study this problem for large families of generalised quadrangles, extending the results from \cite{bagchi} for the generalised quadrangles $W(3,q)$ and $H(3,q^2)$.

\section{Preliminaries}

A {\em partial geometry} with parameters $(s,t, \alpha)$ is an incidence structure of {\em points} and {\em lines} satisfying the following properties.
\begin{enumerate}
\item Through every two distinct points there is at most one line and every two distinct lines meet in at most one point.
\item Every line contains $s+1$ points.
\item Every point lies on $t+1$ lines.
\item For every point $P$ and every line $L$, not through $P$, there are $\alpha$ lines through $P$ which intersect $L$.
\end{enumerate}

A partial geometry with $\alpha=1$ is called a {\em generalised quadrangle}. Generalised quadrangles form an important class of rank two geometries
which are fundamental in the theory of {\em buildings} developed by J. Tits.

The {\em classical finite generalised quadrangles} arise from quadratic and sesquilinear forms and are typically denoted by $Q(4,q)$, $W(3,q)$, $Q(5,q)$, $H(3,q^2)$ and $H(4,q^2)$ (see \cite{payne} for more information). However, other families of generalised quadrangles are known. A large family of examples is given by the {\em translation generalised quadrangles}. The classical generalised quadrangles $Q(4,q)$ and $Q(5,q)$ can be described as a translation generalised quadrangle (see Remark \ref{iso}). Each translation generalised quadrangle can be constructed from a certain set of subspaces in a projective space, called an {\em egg}. This correspondence will be extensively used in our proofs and is detailed below. A projective (affine) space of dimension $n$ over the finite field with $q$ elements is denoted by $\PG(n,q)$ ($\AG(n,q)$).


An {\em egg $\cE_{n,m}$} is a set of $q^m+1$ subspaces of $\PG(2n+m-1,q)$, each of dimension
$(n-1)$,  such that any three different elements of $\cE_{n,m}$ span a  $(3n-1)$-dimensional subspace, and each element $E$ of $\cE_{n,m}$ is contained in an $(n+m-1)$-dimensional subspace, $T_E$, which is disjoint from any element of $\cE_{n,m}\setminus \{E\}$. The subspace $T_E$ is called the {\em tangent space of $\cE_{n,m}$ at $E$}.
The only known examples of eggs are for parameters satisfying $m=n$ and $m=2n$.
Examples with these parameters can be constructed by applying {\em field reduction} (see \cite{LaVa2015}) to the set of points of an oval (for $m=n$) or an ovoid (for $m=2n$). The examples constructed in this way are called {\em elementary}. There exist examples of non-elementary eggs for $m=2n$, but all known examples of eggs for $m=n$ are elementary. For more information about eggs, we refer to \cite{payne,thesismichel}. A complete list of the known examples can be found in \cite[Section 3.8]{thesismichel}.


Each translation generalised quadrangle is isomorphic to the incidence structure $T(\cE)$ constructed from some egg $\cE$ in a $(2n+m-1)$-dimensional projective space $\PG(2n+m-1,q)$ over a finite field $\bF_q$, $q=p^h$, $p$ prime (see e.g. \cite[Theorem 8.7.1]{payne}). In order to construct $T(\cE)$, embed $\PG(2n+m-1,q)$ as a hyperplane $H_\infty$ in $\Pi \cong \PG(2n+m,q)$.

The points of $T(\cE)$ are of three types: 
\begin{enumerate}[(i)]
\item the points in $\Pi\setminus H_\infty$;
\item the $(n+m)$-dimensional subspaces of $\Pi$ intersecting $H_\infty$ in a tangent space of $\cE$;
\item $(\infty)$. \end{enumerate}
The lines of $T(\cE)$ are of two types: 
\begin{enumerate}[(a)]
\item the $n$-dimensional subspace of $\Pi$ intersecting $H_\infty$ in an element of $\cE$; 
\item the elements of $\cE$.
\end{enumerate} Incidence is defined as follows. The point $(\infty)$ is incident with all lines of type $(b)$ and with no line of type $(a)$. The points of type $(ii)$ are incident with the lines of type $(a)$ contained in it, and with the unique line of type $(b)$ contained in it. Points of type $(i)$ are incident with all lines of type $(a)$ containing it. 

It easily follows from the definition that the generalised quadrangle $T(\cE)$ has parameters $(s,t)=(q^n,q^m)$.
\begin{remark}\label{iso} The generalised quadrangle $T(\cE)$ is isomorphic with $Q(4,q)$ if and only if $\cE$ is an elementary egg obtained from a conic, and isomorphic to $Q(5,q)$ if $\cE$ is an elementary egg obtained from an elliptic quadric in a 3-dimensional projective space.

\end{remark}

\section{Linear codes from affine and projective spaces}
Consider the incidence matrix $G$ of points versus $t$-spaces of $\PG(n,q)$, $q=p^h$, $p$ prime, where we index the rows by $t$-spaces and the columns by points. The vector space generated by the rows of $G$ over the finite field $\F_p$ is denoted by $C_t(\PG(n,q))$. 
Similarly, the $p$-ary code generated by the incidence matrix of points versus $t$-spaces (also called {\em $t$-flats}) of $\AG(n,q)$ is denoted by $C_t(\AG(n,q))$. Note that in this paper, we are only considering the {\em $p$-ary} codes where $p$ is the prime such that $q=p^h$.

We will make use of the following classical results (see e.g. \cite[Theorem 5.7.9]{assmus}).
\begin{result}\label{classic}
(1) The minimum weight of $C_t(\PG(n,q))$ is $\frac{q^{t+1}-1}{q-1}$ and the codewords of minimum weight are the scalar multiples of the incidence vectors of $t$-spaces.
(2) The minimum weight of $C_t(\AG(n,q))$ is $q^t$ and the codewords of minimum weight are the scalar multiples of the incidence vectors of $t$-flats.
\end{result}
Result \ref{classic} (2) was established by Delsarte, Goethals and MacWilliams \cite{delsarte} by describing $C_t(\AG(n,q))$ as a {\em subfield subcode} of the generalised Reed-Muller code $\mathcal{R}_q((n-t)(q-1),n)$.
This fact 
allows us to deduce some further properties of $C_t(\AG(n,q))$.
In 2010, Rolland \cite{rolland} determined the second weight of the generalised Reed-Muller codes in almost all cases, including $\mathcal{R}_q((n-t)(q-1),n)$.
Applied to this case, he obtained the following result.
\begin{result}\cite[Theorem 3.8]{rolland}
The next-to-minimum weight of $\mathcal{R}_q((n-t)(q-1),n)$ is $2(q-1)q^{t-1}$ if $q\geq 3$.
\end{result}
For $q=2$, the next-to-minumum weight was already determined in 1974.
\begin{result}\cite{kasami}
The next-to-minimum weight of $\mathcal{R}_2(n-t,n)$ is $2^{t+1}$ if $t=1$ or $t=n-1$ and $2^t+2^{t-1}$ if $1<t<n-1$.
\end{result}

As a corollary, we see that there is a gap in the weight enumerator of $C_t(\AG(n,q))$; in particular, we find that there are no codewords of weight $q^t+1$ in this code.
\begin{cor} \label{gap} There are no codewords of weight $q^t+1$ in $C_t(\AG(n,q))$.
\end{cor}

We end this section with a lemma concerning codewords of the dual code of $C_n(\PG(2n+m,q))$.
Two vectors $(v_1,\ldots,v_m)$ and $(w_1,\ldots,w_m)$ in $(\F_p)^m$ are said to be {\em orthogonal} if their dot product $(v,w)=v_1w_1+\cdots+v_mw_m$ is zero (in $\F_p$). The {\em dual code} of $C_t(\PG(n,q))$, denoted by $C_t(\PG(n,q))^\bot$, is the set of all vectors that are orthogonal to all codewords of $C_t(\PG(n,q))$. 
Observe that a vector $c$ is contained in $C_t(\PG(n,q))^\bot$ if and only if $(c,v)=0$ for all rows $v$ of $G$, where $G$ is the incidence matrix of points versus $t$-spaces of $\PG(n,q)$.

\begin{lemma}\label{lem:codeword_of_dual} Consider a hyperplane $H_\infty$ of $\PG(2n+m,q)$. Let $U$ and $T$ be two $(m+n-1)$-dimensional subspace of $H_\infty$, and $r\in \Pi\setminus H_\infty$. Let $a$ be the incidence vector of $\langle U,r\rangle$ and $b$ the incidence vector of $\langle T,r\rangle$. 
Then  $a-b$ is a codeword of $C_n(\PG(2n+m,q))^\bot$.
\end{lemma}
\begin{proof} Let $\pi$ be an $n$-dimensional subspace of $\PG(2n+m,q)$, then $\pi$ meets both $\langle U,r\rangle$ and $\langle T,r\rangle$ in $1\mod p$ points. Let $v$ be the incidence vector of $\pi$. It follows that $(v,a-b)=(v,a)-(v,b)= 1-1= 0$.
\end{proof}

\section{Linear codes from generalised quadrangles}
\subsection{Linear codes from generalised quadrangles embedded in an affine or projective space}
Let $\mathcal{G}$ be a generalised quadrangle. Consider the incidence matrix $N$ of points versus lines, where we index the rows by lines and the columns by points. The vector space generated by the rows of $G$ over the finite field $\F_p$ is called the {\em $p$-ary code of $\mathcal{G}$}.

The following theorem is a corollary of the results for codes from affine and projective spaces.
\begin{cor}\label{ingebed} Let $\mathcal{G}$ be a generalised quadrangle of order $(s,t)$ which is fully embedded in an affine of projective space of order $q$, $q=p^h$, $p$ prime. Then the minimum weight of the $p$-ary code $C(\mathcal{G})$ is $s+1$ and the codewords of minimum weight are precisely the scalar multiples of the incidence vectors of the lines of $\mathcal{G}$.
\end{cor}
\begin{proof} The code of the embedded generalised quadrangle is a subcode of the code of points and lines of  the ambient space. The corollary follows from Result \ref{classic}.
\end{proof}
\begin{remark} The above corollary shows that the generalised quadrangle $T_2^*(O)$ (which is not a translation generalised quadrangle), where $O$ is a hyperoval, has minimum weight $q-1$ and that all codewords of minimum weight are scalar multiples of incidence vectors of lines of $T_2^*(O)$. The dual of this code has been studied as an LDPC code; the minimum weight was determined in \cite{pepe}.
\end{remark}

\begin{remark}
The above corollary also deals with the classical generalised quadrangles: $W(3,q)$, $Q(4,q)$, $Q(5,q)$, $H(3,q^2)$ and $H(4,q^2)$. In \cite{bagchi}, Bagchi and Sastry derived this result for {\em all regular generalised polygons}. The known regular generalised quadrangles are precisely $W(3,q)$ and $H(3,q^2)$ (see also \cite[Section 3.3]{payne}).
\end{remark}
\subsection{Linear codes from translation generalised quadrangles}

The main theorem of this paper states that the result of Corollary \ref{ingebed} also holds for all translation generalised quadrangles.

In order to prove this, we will put the incidence matrix of $T(\cE)$ in a particular form:

Order the points of $T(\cE)$ such that the points of type $(i)$ come first, then the points of type $(ii)$ and finally the point $(\infty)$. Order the lines of $T(\cE)$ such that the lines of type $(a)$ come before the lines of type $(b)$.
With this ordering of the points and lines of $T(\cE)$ the incidence matrix $N$ has the form
\begin{displaymath}
N=\begin{bmatrix}
  A & B &{\bf 0}\\
  O & D &{\bf 1}\\
\end{bmatrix}
\end{displaymath}
where the matrix $A$ is the incidence matrix of the points of type $(i)$ with lines of type $(a)$; $B$ is the incidence matrix of points of type $(ii)$ with lines of type $(a)$; $O$ is the all zero matrix (a point of type $(i)$ and a line of type $(b)$ are never incident with each other); $D$ is the incidence matrix of points of type $(ii)$ and lines of type $(b)$; and the last column consists of the all-zero column concatenated with the all-one column. Note that $N$ has the following properties:\\
(1) each row of $A$ has weight $q^n$;\\
(2) each row of $B$ has weight $1$;\\
(3) each row of $D$ has weight $q^n$ and each two rows of $D$ have disjoint supports.\\

\begin{theorem}\label{thm:min_wt}
The minimum distance of the $p$-ary code of points and lines of a translation generalised quadrangle of order $(s,t)$, with $s$ a power of the prime $p$, is $s+1$.
\end{theorem}
\begin{proof}  Recall that a translation generalised quadrangle has order $(s,t)=(q^n,q^m)$ for some $n,m$ and $q=p^h$, $p$ prime.
Let $c$ be a codeword in $C(T(\cE))$, say 
\begin{eqnarray}\label{eqn:c}
c=\sum \lambda_i u_i + \sum \mu_i v_i, \quad \quad \lambda_i,\mu_i\in \bF_p
\end{eqnarray}
where the $u_i$'s are incidence vectors of lines $m_i$ of type $(a)$ and the $v_i$'s are incidence vectors of lines $\ell_i$ of type $(b)$ of $T(\cE)$. Suppose $c$ has weight $\leq q^n$. 

{\bf $(i)$} All $\mu_i$'s are zero.

Let $\hat{c}$ denote the codeword in the code $C_n(\AG(2n+m,q))$ of points and $n$-dimensional subspaces of $\AG(2n+m,q)$, consisting of the entries of $c$ in the positions corresponding to the points of type $(i)$ of $T(\cE)$. Then $\hat{c}=\sum \lambda \hat{u}_i$ where the $\hat{u}_i$'s are incidence vectors of $n$-dimensional subspaces $\hat{m}_i$ of $\AG(2n+m,q)$.

If $\wt(\hat{c})>0$ then $\wt(c)=\wt(\hat{c})=q^n$ since the minimum distance of the code $C_n(\AG(2n+m,q))$ is $q^n$ (see Result \ref{classic}(2)). Moreover, by the characterisation of codewords of minimum weight in $C_n(\AG(2n+m,q))$, $\hat{c}$ is a scalar multiple of the incidence vector $v_\pi$ of an $n$-dimensional subspace $\pi$ of $\AG(2n+m,q)$. Since $wt(c)=q^n$,

(*)  for each point $x$ of type $(ii)$, the restriction of the sum $\sum \lambda_i u_i$ to the sum over those $i$ for which $m_i$ is a line of $T(\cE)$ through $x$, is $0$ mod $p$.

Let $\bar{u}_i$ denote the incidence vector of the projective completion $\bar{m}_i$ of the affine subspace $\hat{m}_i$.
By property (*) the sum $\sum \lambda_i \bar{u}_i$ is a codeword, in the code $C_n(\Pi)$ of the points and $n$-dimensional subspaces of the projective space $\Pi=\PG(2n+m,q)$, of weight $q^n$, a contradiction (see Theorem \ref{classic}(1)).

So $\wt(\hat{c})=0$. With the same notation as above, it follows that the linear combination $\bar{c}=\sum \lambda_i\bar{u}_i$ is a codeword in  $C_n(\Pi)$ which is a linear combination of the incidence vectors of a subset $S$ of $\cE$. Since $wt(c)\leq q^n$ there are at most $q^n$ elements in $S$ (ignore $(n-1)$-dimensional subspaces whose incidence vector has a zero coefficient in $\sum \lambda_i \bar{u}_i$). 
Let $E\in S$ and $F\in \cE\setminus S$, and $r\in \Pi\setminus H_\infty$. Let $a$ be the incidence vector of $\langle T_E,r\rangle$ and $b$ the incidence vector of $\langle T_F,r\rangle$. Then, by Lemma \ref{lem:codeword_of_dual}, $a-b$ is a codeword of the dual code of $C_n(\Pi)$. This contradicts $\bar{c}\in C_n(\Pi)$ since by construction the dot product in $\bF_p$ of $\bar{c}$ with $a$ is zero, while the dot product of $\bar{c}$ with $b$ is nonzero.

$(ii)$ All $\lambda_i$'s are zero. 

Since the rows of $N$ containing a row of $D$ have weight $q^n+1$, this implies that there is a linear combination of the rows of $D$, with at least two nonzero coefficients which results in a codeword of weight $\leq q^n$, contradicting property (3) of $N$.

$(iii)$ Finally assume that at least one $\lambda_i$ and at least one $\mu_i$ is nonzero.

If $\wt(\hat{c})>0$ then as before $\wt(\hat c)=q^n$, contradicting the fact that at least one of the $\mu_i$'s is nonzero. 
Hence $\wt(\hat{c})$ must be zero.
Let $\cE=\{E_0,\ldots,E_{q^m}\}$. For each $j\in \{0,\ldots,q^m\}$ denote by $\alpha_j\in \bF_p$ the coefficient of the incidence vector of $E_j\in \cE$ in the linear combination $\sum\lambda_i\bar{u}_i$. 
As in case $(i)$, let $S\subseteq \cE$ denote the set of egg elements $E_j$ for which $\alpha_j\neq 0$.
If $S=\cE$ then $\wt(\sum\lambda_i u_i)\geq q^m+1$ and $c$ has at least one nonzero entry in each of the sets of columns corresponding to the partition defined by the supports of the rows of $D$ (cf. property (3) of $N$).
In order for the codeword $c$ to have weight $\leq q^n$, for at least one such set $T$ of columns of $N$ the entries in $\sum\lambda_i\bar{u}_i$ in the positions corresponding to the columns in $T$ must be a nonzero constant, since otherwise (again using property (3) of $N$) no linear combination of the rows of $D$ can make that part of the codeword zero. This means that the coefficient $\alpha_j=0$ (it is a multiple of $q^n$), where $E_j$ is the element of $\cE$ corresponding to $T$.  This contradicts $S=\cE$.
It follows that $S$ is a proper subset of $\cE$ and we can apply Lemma \ref{lem:codeword_of_dual} and the same argument as in $(i)$ using a codeword  $a-b$ in the dual of the code $C_n(\Pi)$ to obtain a contradiction.

This shows that the minimum weight of $C(T(\cE))$ is at least $q^n+1$. Since each row of $N$ has weight $q^n+1$, the statement follows.
\end{proof}

\begin{theorem}\label{main}
The minimum weight codewords of the $p$-ary code of points and lines of a translation generalised quadrangle $T(\cE)$ of order $(s,t)$, with $s$ a power of the prime $p$, are the incidence vectors of lines of $T(\cE)$.
\end{theorem}
\begin{proof}

We will use the same notation as in the proof of Theorem \ref{thm:min_wt}. Let $c$ be a codeword of $C(T(\cE))$ of weight $q^n+1$ as in (\ref{eqn:c}). 

{\bf $(i)$} All $\mu_i$'s are zero.

If $\wt(\hat{c})>0$ then the weight of $\hat c$ must be $q^n$, since the code $C_n(\AG(2n+m,q))$ does not contain any codewords of weight $q^n+1$ by Corollary \ref{gap}. As in the proof of Theorem \ref{thm:min_wt}, the codeword $\hat{c}$ is a scalar multiple of the incidence vector $v_\pi$ of an $n$-dimensional subspace $\pi$ of $\AG(2n+m,q)$, say $\hat{c}=\lambda v_\pi$. Also, the codeword $c$ has exactly one nonzero entry indexed by a point $x_0$ of type $(ii)$ of $T(\cE)$, equivalently,

(**) the restriction of the sum $\sum \lambda_i u_i$ to the sum over those $i$ for which $m_i$ is a line of $T(\cE)$ through a point $x$ of type $(ii)$, is $0$ mod $p$ for $x\neq x_0$ and nonzero mod $p$ for $x=x_0$.

Consider the projective completion of the subspaces $\hat{m}_i$ and define $w$ by setting $\lambda w=\sum\lambda_i \bar{u}_i$. Then $w$ is a codeword of the code $C_n(\Pi)$. The affine part of the support of $w$ coincides with $\pi$, and by (**) the part of the support of $w$ in $H_\infty$ coincides with $E_0\in \cE$, which is the unique line of type $(b)$ incident with the point $x_0$. Hence $\lambda w$ is a codeword of $C_n(\Pi)$ of minimum weight and must therefore be the scalar multiple of an $n$-dimensional subspace of $\Pi$ (see Result \ref{classic}(1)). It follows that $E_0$ is contained in the projective completion of $\pi$ and that $c$ is the scalar multiple of the incidence vector of a line of type $(a)$ in $T(\cE)$.

If $\wt(\hat{c})=0$ then, as in part $(i)$ of the proof of Theorem \ref{thm:min_wt}, $\bar c$ is a codeword of $C_n(\Pi)$ which is the linear combination of a subset $S$ of $\cE$. 

If $n<m$ then the same argument as in part $(i)$ of the proof of Theorem \ref{thm:min_wt} applies to obtain a contradiction, by using a codeword  $a-b$ in the dual of the code $C_n(\Pi)$ as in Lemma \ref{lem:codeword_of_dual}. 

If $n=m$ then consider two elements $E,F\in \cE$ and $r$ a point of $\Pi\setminus H_\infty$. Let $a$ be the incidence vector of $\langle T_E,r\rangle$ and $b$ the incidence vector of $\langle E,F,r\rangle$. Then again, by Lemma \ref{lem:codeword_of_dual}, $a-b$ is a codeword in the dual code of $C_n(\Pi)$, contradicting the fact that $\bar{c}$ is a codeword of $C(T(\cE))$.

$(ii)$ All $\lambda_i$'s are zero. 

In this case it easily follows from property (3) of the incidence matrix $N$ that $c$ must be the incidence vector of a line of type $(b)$ of $T(\cE)$.

$(iii)$ At least one $\lambda_i$ and at least one $\mu_i$ is nonzero. 

If $\wt(\hat{c})\neq 0$ then $\wt(\hat{c})\geq q^n$ and $\hat{c}$ is the incidence vector of an $n$-dimensional subspace $\hat m$. Let $a$ be the codeword of $C(T(\cE))$ corresponding to the line of type $(a)$ defined by $\hat m$. Then $a-c\in C(T(\cE))$  has weight at most 2, contradicting Theorem \ref{thm:min_wt}.
Hence $\wt(\hat{c})= 0$.

As before, let $\alpha_j\in \bF_p$ denote the coefficient of the incidence vector of $E_j\in \cE$ in the linear combination $\sum\lambda_i\bar{u}_i$, and put $S$ equal to the set of egg elements $E_j$ for which $\alpha_j\neq 0$.

If $n=m$ then consider two elements $E,F\in \cE$ and $r$ a point of $\Pi\setminus H_\infty$. Let $a$ be the incidence vector of $\langle T_E,r\rangle$ and $b$ the incidence vector of $\langle E,F,r\rangle$. Then $a-b$ is a codeword in the dual code of $C_n(\Pi)$, see Lemma \ref{lem:codeword_of_dual}, contradicting the fact that $\bar{c}$ is a codeword of $C(T(\cE))$.

If $n<m$ and $S=\cE$ then $\wt(\sum\lambda_i u_i)\geq q^m+1$ and $c$ has at least one nonzero entry in each of the sets of columns corresponding to the partition defined by the supports of the rows of $D$ (cf. property (3) of $N$).
In order for the codeword $c$ to have weight $q^n+1$, for at least one such set $T$ of columns of $N$ the entries in $\sum\lambda_i\bar{u}_i$ in the positions corresponding to the columns in $T$ must be a nonzero constant, since otherwise (again using property (3) of $N$) no linear combination of the rows of $D$ can make that part of the codeword zero. This means that the coefficient $\alpha_j=0$, where $E_j$ is the element of $\cE$ corresponding to $T$.  This contradicts $S=\cE$.

It follows that $S$ is a proper subset of $\cE$ and we can apply the same argument as in part $(i)$ of the proof of Theorem \ref{thm:min_wt}, using a codeword  $a-b$ in the dual of the code $C_n(\Pi)$, to obtain a final contradiction.
\end{proof}

\section{Locally repairable codes}

In \cite{PaHoOg2013}, the authors study locally repairable codes (LRC's) arising from partial geometries.
We use the following definitions from \cite{PaHoOg2013}. Let $C$ be a code. For every position $i$, the set $\Omega(i)$ is defined to be the set of all parity-check vectors {\em repairing} the $i$-th symbol, i.e. $$\Omega(i)=\{v \in C^\bot: v_i\neq 0\}.$$
\begin{definition} The {\em repair degree} for the $i$-th symbol is $r(i)=min\{wt(v)-1:v\in \Omega(i)\}$ and the overall repair degree $r$ of a linear code of lenght $m$ is its maximum repair degree:
$$r=max_{1\leq i \leq m}\{r(i)\}.$$

\end{definition}
\begin{definition} Let $\Omega_r(i)=\{v\in \Omega_i: wt(v)\leq r+1\}$. The {\em repair availability of $i$} is then $a(i) = |\Omega_r(i)|$, and the overall {\em repair availability of the code} is $$a = min_{1\leq i \leq m}\{a(i)\}.$$
\end{definition}

It was shown in \cite[Theorem 1 and Remark 1]{PaHoOg2013} that the maximum rate of a balanced locally repairable binary code from a partial geometry is achieved when the partial geometry is a generalised quadrangle.

As before, consider the incidence matrix $N$ of points and lines of a partial geometry where rows represent lines and columns represent points. In \cite{PaHoOg2013}, the authors define a {\em $pg-BLRC$-code} as the {\em dual code} of the {\em binary} code generated by $N$.

\begin{result}\cite[Lemma 1]{PaHoOg2013}\label{lrcbound}The repair degree $r$ of binary pg-BLRC's $C$ of a partial geometry with parameters $(s,t,\alpha)$ and its repair availability $a$ satisfies $r\leq s$ and $a\geq t+1$.
\end{result}
Using the results of this paper we can derive the exact values of $r$ and $a$ in the case that the partial geometry is a classical or translation generalised quadrangle or a generalised quadrangle $T_2^*(O)$, $O$ a hyperoval.

\begin{theorem} \label{main2} Let $\mathcal{G}$ be one of the following generalised quadrangles:
a classical generalised quadrangle, i.e. $W(3,q)$, $Q(4,q)$, $Q(5,q)$, $H(3,q^2)$, or $H(4,q^2)$;
$T_2^*(O)$, where $O$ is a hyperoval in $\PG(2,q)$, $q$ even; or
a translation generalised quadrangle of order $(q^n,q^m)$.
The dual code $C$ of the $p$-ary code of $\mathcal{G}$, where $q=p^h$, $p$ prime has repair degree $r=s$ and  repair availability $a=(p-1)(t+1)$. 
\end{theorem}
\begin{proof} We have that $C^\bot$ is the $p$-ary code of the  generalised quadrangle, and we have shown in Corollary \ref{ingebed} and Theorem \ref{main} that the minimum weight of $C^\bot$ is $s+1$. This implies that $r(i)=s$ for all $i$, and hence, $r=s$.
Now $\Omega_i$ consists of the set of all codewords of weight $s+1$ through the point $P_i$ corresponding to $i$-th column, which is, again by Corollary \ref{ingebed} and Theorem \ref{main}, the set of scalar multiples of incidence vectors of the lines through $P_i$. There are $t+1$ lines through $P_i$, each giving rise to $(p-1)$ distinct codewords. Hence $a(i)=(p-1)(t+1)$, for all $i$, and therefore $a=(p-1)(t+1)$.
\end{proof}
\begin{remark} Theorem \ref{main2} shows that the bounds of Result \ref{lrcbound} are sharp for the binary codes of classical and translation generalised quadrangles of even order and for $T_2^*(O)$, $O$ a hyperoval.
\end{remark}



\end{document}